\RequirePackage{fix-cm}
\documentclass[smallextended]{svjour3}       
\smartqed  
\usepackage{graphicx}
\usepackage{subfigure}
\usepackage{enumerate}
\usepackage{xcolor}
\usepackage{amsmath}
\usepackage{amssymb,epsfig,multirow}
\usepackage{float}
\usepackage{color}
\usepackage{booktabs}
\usepackage{enumerate}

\newtheorem{thm}{Theorem}
\newtheorem{prop}{Proposition}
\newtheorem{lem}{Lemma}
\newtheorem{cor}{Corollary}
\newtheorem{rem}{Remark}
\newtheorem{exam}{Example}

\begin{document}
\title{A fast algorithm for globally solving Tikhonov regularized  total least squares problem
\thanks{This research was supported by NSFC under grants 11571029,
11471325 and 11771056, and by fundamental research funds
 for the Central Universities under
grant YWF-17-BJ-Y-52.}
}

\titlerunning{Solving Tikhonov regularized total least squares}        

\author{ Yong Xia  \and Longfei Wang  \and   Meijia Yang }


\institute{
Y. Xia \and L.F. Wang \and M.J. Yang  \at
              State Key Laboratory of Software Development
              Environment, LMIB of the Ministry of Education,
              School of
Mathematics and System Sciences, Beihang University, Beijing,
100191, P. R. China
\email{dearyxia@gmail.com (Y. Xia); kingdflying999@126.com (L.F. Wang);15201644541@163.com (M.J. Yang, Corresponding author)}
 }

\date{Received: date / Accepted: date}

\maketitle

\begin{abstract}
The total least squares problem with the general Tikhonov regularization can be reformulated as a one-dimensional parametric minimization problem (PM), where each parameterized function evaluation corresponds to solving an $n$-dimensional trust region subproblem. Under a mild assumption, the parametric function is differentiable and then an efficient bisection method has been proposed for solving (PM) in literature. In the first part of this paper, we show that the bisection algorithm can be greatly improved by reducing the initially estimated interval covering the optimal parameter.
It is observed that the bisection method cannot  guarantee to find the globally optimal solution since the nonconvex (PM) could have a local non-global minimizer. The main contribution of this paper is to propose an efficient branch-and-bound algorithm for globally solving (PM), based on a novel underestimation of the parametric function over any given interval using only the information of the parametric function evaluations at the two endpoints. We can show that the new algorithm(BTD Algorithm) returns a global $\epsilon$-approximation solution in a computational effort of at most $O(n^3/\sqrt{\epsilon})$ under the same assumption as in the bisection method. The numerical results demonstrate that our new global optimization algorithm performs even much faster than the improved version of the bisection heuristic algorithm.

\end{abstract}

\section{Introduction}
In order to handle the overestimated linear equations $Ax\approx b$ with the noised data matrix $A\in \Bbb R^{m\times n}$ and the noised observed vector $b\in \Bbb R^m$, the total least squares (TLS) approach was firstly proposed in \cite{Golub1980} by solving the following optimization problem:
\begin{equation}
\min_{E\in \Bbb R^{m\times n},r\in \Bbb R^{m},x\in \Bbb R^{n}} \left\{ \|E\|_{\rm F}^2+\|r\|^2:~ (A+E)x=b+r\right\}, \label{Erx}
\end{equation}
where $\|\cdot\|_{\rm F}$ and $\|\cdot\|$ denote the Frobenius norm and the Euclidean norm, respectively, $E$ and $r$ are the perturbations. For more details, we refer to \cite{Golub1996,HL2002,HV1991} and references therein.
Let $(E^*,r^*,x^*)$ be an optimal solution to the above minimization problem (\ref{Erx}). It can be verified that $E^*$ and $r^*$ have a closed-form expression in terms of $x^*$ as the problem (\ref{Erx}) is
a linear-equality constrained convex quadratic program with respect to $E$ and $r$. Therefore, by eliminating $E$ and $r$ from (\ref{Erx}), we obtain the following equivalent quadratic fractional program:
\begin{equation}
\min_{x\in \Bbb R^n}\frac{\|Ax-b\|^2}{\|x\|^2+1},\label{onlyx}
\end{equation}
which can be easily solved by finding the smallest singular value and the corresponding vector of the
augmented matrix $[A~ b]$ if it has a full column rank, see \cite{Golub1980,HV1991}.

For the ill-conditioned (TLS) problem, Tikhonov regularization \cite{T77}  is an efficient way to stabilize the solution by appending a quadratic penalty to the objective function:
\begin{equation}
\min_{E\in \Bbb R^{m\times n},r\in \Bbb R^{m},x\in \Bbb R^{n}} \left\{ \|E\|_{\rm F}^2+\|r\|^2+\rho\|Lx\|^2:~ (A+E)x=b+r \right\},
\label{Erx2}
\end{equation}
where $\rho > 0$ is the penalty parameter, and $L\in \Bbb R^{k\times n}$ $(k\leq n)$ is the particularly chosen regularization matrix of full row-rank. It is worth noting that the model (\ref{Erx2}) also works for the underestimated linear system $Ax=b$. Similar to (\ref{Erx})-(\ref{onlyx}), by eliminating the variables $E$ and $r$, we can recast (\ref{Erx2}) as the following optimization problem with respect to $x$ \cite{B06,JH2013}:
\begin{equation}
({\rm P})~~\min_{x\in \Bbb R^n}\frac{\|Ax-b\|^2}{\|x\|^2+1}+\rho\|Lx\|^2.
\nonumber
\end{equation}
The objective function in (P) is non-convex and has local non-global minimizers.
Consequently, it is difficult to solve (P) to the global optimality.

Let $F\in \Bbb R^{n\times (n-k)}$ be a matrix whose columns form an orthogonal basis of the null space of  $L$. Throughout this paper, we make the following assumption, which was firstly presented in \cite{B06}:
\begin{equation}
{\rm either}~k=n~{\rm or}~\lambda_{\min}\left[ \begin{array}{cc}F^TA^TAF& F^TA^Tb\\b^TAF& \|b\|^2\end{array}\right]<\lambda_{\min}\left(F^TA^TAF\right),\label{as}
\end{equation}
where $\lambda_{\min}(\cdot)$ is the minimal eigenvalue of $(\cdot)$. As shown in \cite{B06}, it is a sufficient condition under which the minimum of (P) is attained. The assumption (\ref{as}) is also essential in an extended version of (P), see \cite{B09}.

It is not difficult to verify that (P) can be equivalently rewritten as the following one-dimensional parametric optimization problem \cite{B06}:
\begin{equation}
(\rm PM)~~\min_{\alpha \geq 1} \left\{\mathcal{G}(\alpha):=\min_{\|x\|^2=\alpha-1}~\left\{ \frac{\|Ax-b\|^2}{\alpha}+\rho \|Lx\|^2 \right\}\right\}, \label{1dim}
\end{equation}
where evaluating the function value $\mathcal{G}(\alpha)$ corresponds to solving an equality version of the trust-region subproblem (TRS) \cite{C00,G80,M83}.
It is shown in \cite{B06} that $\mathcal{G}(\alpha)$ is continuous. Under a mild condition, it is also differentiable. Then, a bisection method is suggested in \cite{B06} to solve the equation $\mathcal{G}'(\alpha)=0$ based on solving a
sequence of (TRS), denoted by Algorithm TRTLSG. It converges to the global minimizer if the function $\mathcal{G}(\alpha)$ is unimodal, which is true  when $L=I$. Since there are exceptional examples \cite{B06} to show that
$\mathcal{G}(\alpha)$ is not always unimodal,
Algorithm TRTLSG remains a heuristic algorithm as it does not guarantee the convergence to the global minimizer of (P).

Let $x^*$ be a globally optimal solution to (P). Then $\alpha^*=\|x^*\|^2+1$ is an optimal solution of $(\rm PM)$. Algorithm TRTLSG starts from an initial interval covering $\alpha^*$, denoted by $[\alpha_{\min},\alpha_{\max}]$. As in \cite{B06}, $\alpha_{\min}$ is trivially set as $1+\epsilon_1$, where $\epsilon_1>0$ is a tolerance parameter. $\alpha_{\max}$ is chosen in a closed form based on
a tedious derivation of the upper bound of $\|x^*\|$
under Assumption (\ref{as}).
Notice that the computational cost of Algorithm TRTLSG is proportional to $log(\alpha_{\max}-\alpha_{\min})$, the length of the initial interval.
Thus, in the first part of this paper, we try to improve the lower and upper estimations of $\alpha^*$. More precisely, we firstly establish a new closed-form upper bound of $\alpha^*$, which greatly improves the quality of the current estimation at the same computational cost. Secondly, a new lower bound of $\alpha^*$ is derived in place of the trivial setting $\alpha_{\min}=1+\epsilon_1$. With the new setting of $\alpha_{\min}$ and $\alpha_{\max}$, the efficiency of Algorithm TRTLSG is greatly improved for our tested numerical result.

The main contribution of this paper is to propose a novel two-layer dual  approach for underestimating $\mathcal{G}(\alpha)$ over any given interval, without additional computational cost except for  evaluating $\mathcal{G}(\alpha)$ at the two endpoints of the interval.
With this high-quality underestimation, we develop an
efficient branch-and-bound algorithm to solve the one-dimensional parametric reformulation $(\rm PM)$ (\ref{1dim}). Our new algorithm guarantees to find a global $\epsilon$-approximation solution of $(\rm PM)$ in at most $O(1/\epsilon)$ iterations and the computational effort in each iteration is $O(n^3\log(1/\epsilon))$. Under the additional assumption to make $\mathcal{G}(\alpha)$ be differentiable, the number of iterations can be further reduced to $O(1/\sqrt{\epsilon})$. Numerical results demonstrate that, in most cases, our new global optimization algorithm is much faster than the improved version of the heuristic Algorithm TRTLSG.

The remainder of the paper is organized as follows.
In Section 2, we present some preliminaries and the bisection heuristic Algorithm TRTLSG. In Section 3, we establish new lower and upper bounds on the norm of any optimal solution of (P), with which the computation cost of Algorithm TRTLSG greatly decreases.  In Section 4, we propose a novel underestimation and then use it to develop an efficient branch-and-bound algorithm. The worst-case computational complexity is also analyzed. Numerical comparisons among the above three algorithms are reported in Section 5. Concluding remarks are made in Section 6.

Throughout the paper, the notation ``:='' denotes ``define''. $v(\cdot)$ denotes the optimal objective value of the
problem $(\cdot)$. $I$ is the identity matrix. The matrix $A\succ(\succeq)0$ stands for that $A$ is positive
(semi-)definite. The inner product of two matrices $A$ and $B$ are tr($AB^T$). ${\rm Range}(A)=\{Ax: x\in\Bbb R^n\}$ is the range space of $A$.
The one-dimensional intervals $\{x: a< x < b\}$ and $\{x: a \le x \le
b\}$ are denoted by $(a, b)$ and $[a, b]$, respectively. $\lceil(\cdot)\rceil$ is the smallest integer larger than or equal to $(\cdot)$.

\section{The bisection algorithm}
In this section, we present the bisection algorithm, denoted by Algorithm TRTLSG in \cite{B06}. To begin with, we firstly list some preliminary results of (P) and $\mathcal{G}(\alpha)$ defined in (\ref{1dim}).

\begin{thm}[\cite{B06}]\label{thm:condition}
Under Assumption (\ref{as}) and $k<n$, we have
\begin{equation}
 v({\rm P})\leq \lambda_{\min}\left[ \begin{array}{cc}F^TA^TAF& F^TA^Tb\\b^TAF& \|b\|^2\end{array}\right] \label{ub1}
\end{equation}
and the minimum of {\rm (P)} is attained.
\end{thm}

\begin{thm}[\cite{B06}]\label{thm:ub1}
Let $x^*$ be an optimal solution of {\rm(P)}.
If $k=n$, we have
\[
\|x^*\|^2 \leq \frac{\|b\|^2}{\rho \cdot\lambda_{\min}\left(LL^T\right)}.
\]
Otherwise, if $k<n$, under Assumption (\ref{as}), it holds that
\begin{equation}
\|x^*\|^2 \leq  \max \bigg\{ 1, \frac{\|b\|^2+ \left(\lambda_{\max}(A^TA)+\|A^{T}b\|\right)(\delta +2\sqrt{\delta})+l_1(1+\delta)}{l_1-l_2} \bigg\}^2+\delta, \label{ubb2}
\end{equation}
where $\lambda_{\max}(\cdot)$ is the maximal eigenvalue of $(\cdot)$, and
\begin{eqnarray}
l_1&=& \lambda_{\min}\left(  F^TA^TAF \right), \label{L1}\\
l_2&=& \lambda_{\min}\left[ \begin{array}{cc}F^TA^TAF& F^TA^Tb\\b^TAF& \|b\|^2\end{array}\right],\label{L2}\\
\delta&=& \frac{l_2}{\rho\cdot\lambda_{\min}\left(LL^T\right)}.\nonumber
\end{eqnarray}
\end{thm}

Define
\begin{equation}
Q_{\alpha}:=\frac{1}{\alpha}A^{T}A+\rho L^{T}L,~~f_{\alpha}:=\frac{1}{\alpha}A^{T}b.\label{Qf}
\end{equation}
We reformulate $\mathcal{G}(\alpha)$ (\ref{1dim}) as
\begin{equation}
\mathcal{G}(\alpha) = \min_{\|x\|^2=\alpha-1} \left\{ x^{T}Q_{\alpha}x-2f_{\alpha}^{T}x+\frac{\|b\|^2}{\alpha} \right\}, \label{1dim2}
\end{equation}
which is an equality version of the trust region subproblem (TRS) \cite{C00,G80,M83}. Though it is a non-convex optimization problem, there is a
necessary and sufficient condition to characterize the globally optimal solution of (TRS) (\ref{1dim2}). It means that (TRS) enjoys the hidden convexity.
\begin{thm}[\cite{F04,G80,M83}]\label{thm:trs}
For any $\alpha>1$,
$x(\alpha)$ is an optimal solution of (\ref{1dim2}) if and only if there exists $\lambda(\alpha)\in \Bbb R$ such that
\begin{eqnarray}
& (Q_{\alpha}-\lambda(\alpha)I)x(\alpha)=f_{\alpha}, \label{KKT1}\\
& \|x(\alpha)\|^2= \alpha-1, \label{KKT2}\\
& Q_{\alpha}-\lambda(\alpha)I \succeq 0. \label{KKT3}
\end{eqnarray}
\end{thm}
\begin{cor}
For any $\alpha>1$, suppose
\begin{equation}
f_{\alpha} \notin {\rm Null}(Q_{\alpha}-\lambda_{\min}(Q_{\alpha})I)^\bot,\label{uniq}
\end{equation}
then the KKT conditions (\ref{KKT1})-(\ref{KKT3}) has a unique solution ($x(\alpha)$, $\lambda(\alpha)$).
\end{cor}
Theorem \ref{thm:trs} supported many algorithms for solving (TRS), see, for example, \cite{C00,F04,M83,PW,R97,S97}. In this paper, for the tested medium-scale problems, we apply the solution approach based on the
complete spectral decomposition \cite{G89}.

\begin{thm}[\cite{B06}]
$\mathcal{G}(\alpha)$ is continuous over $[1,+\infty)$.
\end{thm}

\begin{thm}[\cite{B06}]\label{grad}
Suppose that assumption (\ref{uniq}) holds for all $\alpha>1$, then $\mathcal{G}(\alpha)$ is differentiable of any order. Moreover, the first derivative is given by
\begin{equation}
\mathcal{G}'(\alpha)=\lambda(\alpha)-\frac{\|Ax(\alpha)-b\|^2}{\alpha^2},
\nonumber
\end{equation}
where ($x(\alpha)$, $\lambda(\alpha)$) is the unique solution of the KKT conditions (\ref{KKT1})-(\ref{KKT3}).
\end{thm}

Based on Theorems \ref{thm:ub1}, \ref{thm:trs} and \ref{grad},  applying the simple bisection method to solve $\mathcal{G}'(\alpha)=0$ yields  Algorithm TRTLSG proposed in \cite{B06}.

\begin{center}
\fbox{\shortstack[l]{
{\bf Algorithm TRTLSG \cite{B06}}\\
1.~Input: $A\in \Bbb R^{m\times n}$, $b\in \Bbb R^m$, $L\in \Bbb R^{k\times n}$, $\rho> 0$,\\ ~~~~~~~~~~~~and $\epsilon_1>0,~\epsilon_2>0$: tolerance parameters.\\
2.~Set $\alpha_{\min}:= 1+\epsilon_1$. Let $\alpha_{\max}$ be the upper bound given in Theorem \ref{thm:ub1}.\\
3.~While $|\alpha_{\max}- \alpha_{\min}| > \epsilon_2$, repeat the following steps (a)-(c):\\
~~~~(a) Set $\alpha:= \frac{\alpha_{\min}+ \alpha_{\max}}{2}$.\\
    ~~~~(b) Solve (TRS) (\ref{1dim2}) or the equivalent KKT conditions (\ref{KKT1})-(\ref{KKT3}) \\
    ~~~~~~~~~and obtain the solution $(x(\alpha),\lambda({\alpha}))$.\\
    ~~~~(c) If $\lambda({\alpha})- \frac{\|Ax(\alpha)-b\|^2}{\alpha^2}>0$, then set $\alpha_{\max}:= \alpha$; else set $\alpha_{\min}:= \alpha$. \\
4.~Output $x^*:=x(\alpha_{\max})$: an approximately optimal solution of (P).
}}
\end{center}

If the function $\mathcal{G}(\alpha)$ is unimodal,
Algorithm TRTLSG converges to the global minimizer of (P).
It is proved to be true when $L=I$ \cite{B06}. In general,
this is not true.
A counterexample (with
$m = n =4$, $k = 3$) is plotted in \cite{B06} to show that $\mathcal{G}(\alpha)$ is not always unimodal. Thus, Algorithm TRTLSG could
return a local non-global minimizer of (P).

\section{Bounds on the norm of any globally optimal solution}
In this section, we establish new lower and upper bounds on
the norm of any globally optimal solution of (P). They help to greatly improve the efficiency of Algorithm TRTLSG.

\subsection{A new lower bound}
To our best knowledge, there is no nontrivial lower bound on the norm of any globally optimal solution of (P) except for the trivial setting $1+\epsilon_1$ in \cite{B06}. In this subsection, in order to derive such a new lower bound, we firstly need a technical lemma.
\begin{lem}\label{lem1}
Under Assumption (\ref{as}), for any $\mu>0$, we have
\begin{equation}
A^TA+\mu L^TL\succ 0. \label{lem0}
\end{equation}
\end{lem}
\begin{proof}
It is sufficient to consider the nontrivial case $k<n$, as the other case $k=n$ implies that $L^TL\succ 0$ and hence (\ref{lem0}) holds true. Then, according to Assumption (\ref{as}), we have
\[
\lambda_{\min}\left(F^TA^TAF\right)>\lambda_{\min}\left[ \begin{array}{cc}F^TA^TAF& F^TA^Tb\\b^TAF& \|b\|^2\end{array}\right]=\lambda_{\min}\left([AF~b]^T [AF~b]\right)\ge 0.
\]
It follows that $F^TA^TAF\succ 0$, that is,
$y^TF^TA^TAFy>0$ for all $y\neq 0$. Since $\{Fy:y\in \Bbb R^{n-k},y\neq 0\}=\{x\neq 0:Lx=0\}$, we have
 $x^TA^TAx>0$ for all $Lx=0$ and $x\neq 0$. This implies that $Null(L)\cap Null(A)= \{0\}$, which means
that for any $x \neq 0$, either $x^TA^TAx>0$ or $x^TL^TLx>0$. Thus, $x^T(A^TA+\mu L^TL)x>0$ for all $x \neq 0$, implying that $A^TA+\mu L^TL \succ 0$.

The proof is complete.
\end{proof}

\begin{thm}\label{thm:lowb}
Suppose $A^Tb\neq 0$.
Let $x^*$ be an optimal solution of {\rm(P)}. Define
\[
\kappa_1= \left\{\begin{array}{ll} \|b\|^2-b^TA(A^{T}A+ \rho L^{T}L)^{-1}A^{T}b, &{\rm if}~k=n,\\
\min\left\{l_2,\|b\|^2-b^TA(A^{T}A+ \rho L^{T}L)^{-1}A^{T}b\right\}, &{\rm if}~k<n,\end{array}\right.
 \]
where $l_2$ is defined in (\ref{L2}). $\kappa_2=\lambda_{\min}(A^{T}A+ \rho L^{T}L)-\kappa_1$. Then, $\kappa_1<\|b\|^2$ and
\begin{equation}
\|x^*\|\geq \left\{\begin{array}{ll}\frac{\|b\|^2-\kappa_1}{2\|A^{T}b\|}, &{\rm if}~ \kappa_2=0,\\
 \frac{\|A^{T}b\|-\sqrt{\|A^{T}b\|^2-\kappa_2 (\|b\|^2-\kappa_1)}}{\kappa_2}, &{\rm otherwise}.\end{array}\right. \label{xlow}
\end{equation}
\end{thm}
\begin{proof}
Since
\[
\dfrac{\|Ax-b\|^2}{\|x\|^2+1}+\rho \|Lx\|^2
\leq J(x):= \|Ax-b\|^2 + \rho \|Lx\|^2,
\]
we have
\[
v({\rm P})\le \min_{x\in \Bbb R^n} J(x).
\]
By Lemma \ref{lem1}, $J(x)$ has a unique minimizer $x^*=(A^{T}A+ \rho L^{T}L)^{-1}A^{T}b$. Since  $A^{T}b \neq 0$, we have $x^*\neq 0$ and thus it holds that
\[
J(x^*)=\|b\|^2-b^TA(A^{T}A+ \rho L^{T}L)^{-1}A^{T}b<J(0)=\|b\|^2.
\]
We obtain that $\kappa_1<\|b\|^2$.

According to Theorem \ref{thm:condition} and the definitions of $\kappa_1$, $l_2$ and $x^*$, we have
\begin{eqnarray}
\kappa_1\ge v({\rm P})&\ge& \frac{\|Ax^*-b\|^2 + \rho \|Lx^*\|^2}{\|x^*\|^2+1}\nonumber\\
&\ge& \frac{\lambda_{\min}(A^{T}A+ \rho L^{T}L)\|x^*\|^2-2b^TAx^*+\|b\|^2}{\|x^*\|^2+1}\nonumber\\
&\ge& \frac{\lambda_{\min}(A^{T}A+ \rho L^{T}L)\|x^*\|^2-2\|A^{T}b\|\|x^*\|+\|b\|^2}{\|x^*\|^2+1}\nonumber,
\end{eqnarray}
where the last inequality follows from Cauchy-Schwartz inequality.
Therefore, we obtain
\begin{equation}
\left(\lambda_{\min}(A^{T}A+ \rho L^{T}L)-\kappa_1\right)\|x^*\|^2 - 2\|A^{T}b\| \|x^*\| + \|b\|^2-\kappa_1 \leq 0. \label{lbd1}
\end{equation}
Solving the quadratic inequality (\ref{lbd1}) with respect to $\|x^*\|$ gives the lower bound on $\|x^*\|$ (\ref{xlow}). The proof is complete.
\end{proof}

Suppose $A^{T}b = 0$, (P) is reduced to
\begin{equation}
 \min_{x\in \Bbb R^n}\frac{x^TA^TAx+\|b\|^2}{\|x\|^2+1}+\rho\|Lx\|^2,
\label{TR2}
\end{equation}
If $b=0$,  since the objective function (\ref{TR2}) is nonnegative, we can see that $x^*=0$ is an optimal solution of (P).
For this case,  the initial setting $\alpha_{\min}= 1+\epsilon_1$ in Algorithm TRTLSG \cite{B06} is overestimated.

Finally, we assume  $A^{T}b = 0$ and  $b\neq 0$. The relation between the initial setting of $\alpha_{\min}$ and the quality of the approximation minimizer of $\mathcal{G}(\alpha)$ over $\{1\}\cup [\alpha_{\min},\alpha_{\max}]$  is established as follows.
\begin{prop}\label{prop1}
Suppose  $A^{T}b = 0$ and  $b\neq 0$, for any $\epsilon\ge0$,
\[
\min\left\{\mathcal{G}(1),~\min_{ \alpha \ge \frac{\|b\|^2}{\|b\|^2- \epsilon}} \mathcal{G}(\alpha)\right\} \le v({\rm P})+\epsilon.
\]
\end{prop}
\begin{proof}
According to the definition (\ref{1dim}), we have
\begin{eqnarray*}
\min_{1\le\alpha \le \frac{\|b\|^2}{\|b\|^2- \epsilon}} \mathcal{G}(\alpha) &\ge &
 \min_x \frac{x^TA^TAx+\|b\|^2}{ \|b\|^2/(\|b\|^2- \epsilon)}+\rho\|Lx\|^2 \\
& \geq & \frac{\|b\|^2}{\|b\|^2/(\|b\|^2- \epsilon)} = \|b\|^2- \epsilon.
\end{eqnarray*}
Since $\mathcal{G}(1)=\|b\|^2$, we have
\begin{eqnarray*}
v({\rm P})&=&\min\left\{\min_{1\le\alpha \le \frac{\|b\|^2}{\|b\|^2- \epsilon}} \mathcal{G}(\alpha),~\min_{ \alpha \ge \frac{\|b\|^2}{\|b\|^2- \epsilon}} \mathcal{G}(\alpha)\right\} \\
&\ge& \min\left\{\mathcal{G}(1)-\epsilon,~\min_{ \alpha \ge \frac{\|b\|^2}{\|b\|^2- \epsilon}} \mathcal{G}(\alpha)\right\}\ge \min\left\{\mathcal{G}(1),~\min_{ \alpha \ge \frac{\|b\|^2}{\|b\|^2- \epsilon}} \mathcal{G}(\alpha)\right\}-\epsilon.
\end{eqnarray*}
\end{proof}

\subsection{New upper bounds}
In this subsection, we propose two improved upper bounds on the norm of any optimal solution of (P), one of which has the same computational cost as  the upper bound given in Theorem \ref{thm:ub1}.

Let $x^*$ be any globally optimal solution of (P).  Consider the nontrivial case $k<n$. Though the derivation of
the upper bound (\ref{ubb2}) given in Theorem \ref{thm:ub1} is rather tedious,
it is basically based on the two inequalities
\begin{eqnarray}
&&\|Ax^*-b\|^2\le l_2(\|x^*\|^2+1),\label{in1}\\
&&\rho \|Lx^*\|^2 \le l_2,\label{in2}
\end{eqnarray}
which follow from (\ref{ub1}) in Theorem \ref{thm:condition}. Thus, a tighter upper bound is given by
\begin{equation}
\max_{(\ref{in1}),(\ref{in2})}~\|x^*\|^2.\label{qcqp}
\end{equation}
It leads to an inhomogeneous quadratic constrained quadratic program and still hard to solve. We further relax (\ref{qcqp}) to its Lagrangian dual problem, which can be rewritten as the following semidefinite program (SDP):
\begin{eqnarray*}
({\rm SDP})~~&\min & ~  t \\
               &{\rm s.t.} &   \mu_1 B_2+\mu_2 B_3-B_1\succeq 0,\\
               &  &\mu_1 \geq 0, ~\mu_2 \geq 0,
\end{eqnarray*}
where
\[
B_1=\left(\begin{matrix}I& 0\\ 0& -t \end{matrix}\right),~~ B_2=\left(\begin{matrix} A^{T}A-l_2 I& -A^{T}b \\ -b^{T}A & b^{T}b- l_2 \end{matrix}\right),~~
B_3=\left(\begin{matrix} \rho L^{T}L & 0 \\ 0 & -l_2 \end{matrix}\right).
\]
$v$(SDP) gives a new upper bound of $\|x^*\|^2$. If strong duality holds for (\ref{qcqp}), then the new bound $v$(SDP) is definitely not weaker than (\ref{ubb2}).  But the computation of (SDP) is much more time-consuming than that of (\ref{ubb2}).

In the following, we propose a new upper bound of $\|x^*\|^2$ with the same computational effort as (\ref{ubb2}). The basic idea is directly following the original inequality (\ref{ub1}) rather than (\ref{in1})-(\ref{in2}).

\begin{thm}\label{thm:newub}
Let $\beta=2\lambda_{\max}(A^{T}A)$, $\gamma=2 \|A^{T}b\|$, $\zeta= \rho \lambda_{\min}(L L^{T})$. We have
\begin{eqnarray}
\|x^*\|^2 &\leq& -\frac{1}{2} +\frac{l_2}{2\zeta} +\frac{\sqrt { (\zeta-l_2)^2+\beta^2+4\zeta l_2+\frac{\gamma^2}{l_1-l_2}\zeta } }{2\zeta} \nonumber\\
&&+ \left( \frac{\gamma+\sqrt{\gamma^2 +(l_1-l_2)(4l_2+\frac{\beta^2}{\zeta} +\frac{(\zeta-l_2)^2}{\zeta}) } }{2(l_1-l_2)} \right)^2. \label{newub2}
\end{eqnarray}
\end{thm}
\begin{proof}
$x^*$ has  the following decomposition
\begin{equation}
x^*=L^Tw+Fv, \label{dec}
\end{equation}
where $w\in\Bbb R^k$ and $v\in \Bbb R^{n-k}$. Substituting (\ref{dec}) into (\ref{ub1}) yields
\begin{equation}
\|AL^Tw+AFv-b\|^2+\rho \|LL^Tw\|^2(t_1+t_2+1)\le l_2(t_1+t_2+1), \label{ineq}
\end{equation}
where $t_1=\|L^Tw\|^2$ and $t_2=\|v\|^2$.
Since
\begin{eqnarray}
\|AL^Tw+AFv-b\|^2&=& \|AFv\|^2+2v^TF^TA^T(AL^Tw-b)+\|AL^Tw-b\|^2\nonumber\\
&\ge &l_1\|v\|^2+2v^TF^TA^T(AL^Tw-b)\nonumber\\
&\ge &l_1\|v\|^2-2\|v\|\cdot\|A^TAL^Tw-A^Tb\|\nonumber\\
&\ge &l_1\|v\|^2-2\|v\|\left(\lambda_{\max}(A^TA)\|L^Tw\|+\|A^Tb\|\right)\nonumber\\
&=&l_1t_2-\sqrt{t_2}\left(\beta\sqrt{t_1}+\gamma\right),\nonumber
\end{eqnarray}
where the second inequality follows from Cauchy-Schwartz inequality, and
\begin{eqnarray}
\rho\|LL^Tw\|^2&=&\rho w^T(LL^T)^{\frac{1}{2}}LL^T(LL^T)^{\frac{1}{2}}w \nonumber\\
&\ge&\rho \lambda_{\min}(LL^T) w^T(LL^T)^{\frac{1}{2}}(LL^T)^{\frac{1}{2}}w\nonumber\\
&=&\zeta w^TLL^Tw\nonumber\\
&=&\zeta t_1,\nonumber
\end{eqnarray}
it follows from the inequality (\ref{ineq}) that
\[
l_1t_2-\sqrt{t_2}\left(\beta\sqrt{t_1}+\gamma\right)+\zeta t_1(t_1+t_2+1)\le l_2(t_1+t_2+1).
\]
Or equivalently, we have
\begin{equation}
\left[\zeta t_1^2+(\zeta-l_2) t_1\right]+\left[\zeta t_1t_2-\beta\sqrt{t_1t_2}\right]
+\left[(l_1-l_2)t_2-\gamma\sqrt{t_2}\right]-l_2\le 0.\label{ineq2}
\end{equation}
Notice that
\begin{eqnarray}
&&\zeta t_1^2+(\zeta-l_2) t_1\ge -\frac{(\zeta-l_2)^2}{4\zeta},\label{i:1}\\
&&\zeta t_1t_2-\beta\sqrt{t_1t_2}\ge -\frac{\beta^2}{4\zeta},\label{i:2}\\
&&(l_1-l_2)t_2-\gamma\sqrt{t_2}\ge -\frac{\gamma^2}{4(l_1-l_2)}.\label{i:3}
\end{eqnarray}
Substituting (\ref{i:2})-(\ref{i:3}) into (\ref{ineq2}), we obtain
\[
t_1\le -\frac{1}{2} +\frac{l_2}{2\zeta} +\frac{\sqrt { (\zeta-l_2)^2+\beta^2+4\zeta l_2+\frac{\gamma^2}{l_1-l_2}\zeta } }{2\zeta}.
\]
Similarly, substituting (\ref{i:1})-(\ref{i:2}) into (\ref{ineq2}), we have
\[
\sqrt{t_2}\le \frac{\gamma+\sqrt{\gamma^2 +(l_1-l_2)(4l_2+\frac{\beta^2}{\zeta} +\frac{(\zeta-l_2)^2}{\zeta}) } }{2(l_1-l_2)}.
\]
The proof is complete as $\|x^*\|^2=t_1+t_2$.
\end{proof}

In order to compare the existing upper bound (\ref{ubb2}) with the new bounds $v$(SDP) and (\ref{newub2}), we do numerical experiments using the noise-free data of the first example presented in Section 6. The dimension $n$ varies from $20$ to $3000$ and the regularization parameter $\rho$ is simply fixed at $0.5$. The computational environment is presented in Section 6. We report the numerical results  in Table \ref{tab:1}. It can be seen that $v$(SDP) gives the tightest upper bound with the highest computation cost. For each test instance, the new upper bound  (\ref{newub2}) is much tighter than the existing upper bound (\ref{ubb2}) in the same computational time. We can see that, for the instance of dimension $1000$, Algorithm TRTLSG will save  $\log_2\left(\frac{1.97\times10^{12}}{4.79\times10^6}\right)\approx 15$ iterations if the new upper bound (\ref{newub2}) is used to replace (\ref{ubb2}).
From Columns 2-3 of Table \ref{tab:1}, it is observed that the new lower bound (\ref{xlow}), solved at a low computational cost, is much tighter than the trivial bound $1+\epsilon_1=1.1$.
For the instance of dimension $1000$,  replacing the trivial bound $1+\epsilon_1$ with
the new lower bound (\ref{xlow}) will help
Algorithm TRTLSG to save  $\log_2\left(\frac{1.64\times10^{2}}{1.1}\right)\approx 7$ iterations.

\begin{table}[h]\centering
\caption{
Computational time (in seconds) and the quality of the new lower bound (\ref{xlow}), the upper bound (\ref{ubb2}) given in \cite{B06},
the new upper bounds (\ref{newub2}) and $v$(SDP), where $aeb=a\times 10^b$.} \label{tab:1}
\begin{tabular}[h]{rrrrrrrrr}
\hline
&\multicolumn{2}{c}{new b.d. (\ref{xlow})}&\multicolumn{2}{c}{b.d. (\ref{ubb2}) in \cite{B06}}&\multicolumn{2}{c}{new b.d. (\ref{newub2})}&\multicolumn{2}{c}{$v$(SDP)}\\
\cline{2-9}
	n&	time	&$\alpha_{\min}$&time	&$\alpha_{\max}$&time	&$\alpha_{\max}$&	time	 &$\alpha_{\max}$	\\
\hline
20	&	0.00 	&	4.28 	&	0.00 	& $3.02e4$  &	0.00 	& $2.28e3$  &	1.39 	&	 $4.76e1$	\\
50  & 0.00      & 9.18      & 0.00      & $1.35e6$  & 0.00      & $1.32e4$  & 0.47      & $1.60e2$ \\
100 & 0.00      & $1.73e1$  & 0.00      & $3.08e7$  & 0.00      & $5.08e4$  & 0.63      & $4.27e2$ \\
200 & 0.00      & $3.37e1$  & 0.00      & $7.98e8$  & 0.02      & $1.98e5$  & 1.30      & $1.20e3$ \\
500 & 0.03      & $8.27e1$  & 0.02      & $6.62e10$ & 0.03      & $1.21e6$  & 6.97      & $5.17e3$ \\
1000 & 0.09     & $1.64e2$  & 0.09      & $1.97e12$ & 0.09      & $4.79e6$  & 25.75     & $1.66e4$ \\
1200 & 0.14     & $1.97e2$  & 0.13      & $4.83e12$ & 0.14      & $6.88e6$  & 59.03     & $2.28e4$ \\
1500 & 0.22     & $2.46e2$  & 0.20      & $1.45e13$ & 0.20      & $1.07e7$  & 102.81    & $3.37e4$ \\
1800 & 0.31     & $2.95e2$  & 0.33      & $3.56e13$ & 0.33       & $1.54e7$  &149.72    & $4.66e4$ \\
2000 & 0.39     & $3.28e2$  & 0.41      & $6.00e13$ & 0.41       & $1.90e7$  &199.24    & $5.63e4$ \\
2500 & 0.90     & $4.10e2$  & 0.98      & $1.81e14$ & 1.02       & $2.96e7$  &307.04    & $8.42e4$ \\
3000 & 2.42     & $4.92e2$  & 2.44      & $4.46e14$ & 2.36       & $4.26e7$  &470.16    & $1.17e5$ \\
\hline\hline
\end{tabular}
\end{table}

\section{Branch-and-bound algorithm based on a new bound}
In this section we firstly present a new two-layer dual approach for underestimating $\mathcal{G}(\alpha)$ (\ref{1dim}) and then use it to develop an efficient branch-and-bound algorithm(BTD Algorithm,). The worst-case computational complexity is also analyzed.

\subsection{A new underestimation approach}
The efficiency to solve (P) via (\ref{1dim}) relies on an easy-to-compute and high-quality lower bound of $\mathcal{G}(\alpha)$ (\ref{1dim}) over any given interval $[\alpha_i,\alpha_{i+1}]$. The difficulty is that there seems to be no closed-form expression of $\mathcal{G}(\alpha)$. In this subsection, we present a new approach for underestimating $\mathcal{G}(\alpha)$.

For the sake of simplicity, let $p(x):=\|Ax-b\|^2$, $g(x):=\|x\|^2+1$, and $h(x):=\rho \|Lx\|^2$. Our goal is to find a lower bound of the problem:
\begin{equation}
\min_{\alpha\in[\alpha_i,\alpha_{i+1}]} \left\{\mathcal{G}(\alpha)=\min_{g(x)=\alpha} \frac{p(x)}{\alpha}+h(x)\right\}
\label{newb0}
\end{equation}
using only the solutions of evaluating $\mathcal{G}(\alpha)$ at the two endpoints $\alpha_i$ and $\alpha_{i+1}$, i.e., $(x(\alpha_i),\lambda(\alpha_i))$ and $(x(\alpha_{i+1}),\lambda(\alpha_{i+1}))$, which are obtained by solving (\ref{KKT1})-(\ref{KKT3}) with the setting $\alpha=\alpha_i$ and $\alpha=\alpha_{i+1}$, respectively.

For any $\alpha\in[\alpha_i,\alpha_{i+1}]$, the  problem of evaluating $\mathcal{G}(\alpha)$ is an equality version of (TRS) and enjoys the strong Lagrangian duality \cite{XSR16}. Then, it follows that
\begin{eqnarray}
\mathcal{G}(\alpha)&=& \max_{\lambda\in\Bbb R}\min_{x\in\Bbb R^n}  \frac{p(x)}{\alpha}+h(x)-\lambda(g(x)-\alpha)\label{G1}\\
&=&\max_{\lambda\in\Bbb R}   \frac{p(x(\lambda,\alpha))}{\alpha}+h(x(\lambda,\alpha))-\lambda g(x(\lambda,\alpha))+\alpha\lambda,\label{G2}
 \end{eqnarray}
where $x(\lambda,\alpha)$ is an optimal solution of the inner minimization of (\ref{G1}).

Let $(x(\alpha),\lambda(\alpha))$ be the solution of the KKT system (\ref{KKT1})-(\ref{KKT3}), i.e., $x(\alpha)$ is an optimal solution to the minimization problem of evaluating $\mathcal{G}(\alpha)$ and $\lambda(\alpha)$ is the Lagrangian multiplier corresponding to the sphere constraint $g(x)-\alpha=0$. Then, we have
\begin{eqnarray}
\mathcal{G}(\alpha)
&=& \frac{p(x(\alpha))}{\alpha}+h(x(\alpha))\label{G11}\\
&=& \frac{p(x(\alpha))}{\alpha}+h(x(\alpha))-\lambda(\alpha)
(g(x(\alpha))-\alpha)\nonumber\\
&=& \min_{x\in\Bbb R^n}   \frac{p(x)}{\alpha}+h(x)-\lambda(\alpha)(g(x)-\alpha),\label{G3}
\end{eqnarray}
where (\ref{G11}) is due to the constraint that $g(x)-\alpha=0$, (\ref{G3}) follows from the fact that $\lambda(\alpha)$ is an optimal solution to the outer optimization problem of (\ref{G1}).

Setting $\alpha=\alpha_i$ and $\alpha=\alpha_{i+1}$ in (33), respectively, we have
\begin{eqnarray}
&&\frac{p(x(\lambda,\alpha))}{\alpha_i}+h(x(\lambda,\alpha)) - \lambda(\alpha_i) g(x(\lambda,\alpha))+\alpha_i \lambda(\alpha_i) \geq
 \mathcal{G}(\alpha_i), \nonumber
\end{eqnarray}
and
\begin{eqnarray}
&&   \frac{p(x(\lambda,\alpha))}{\alpha_{i+1}}+h(x(\lambda,\alpha)) - \lambda(\alpha_{i+1}) g(x(\lambda,\alpha))+\alpha_{i+1} \lambda(\alpha_{i+1})
\geq
\mathcal{G}(\alpha_{i+1}), \nonumber
\end{eqnarray}
which give hints of estimating the unknowns in (\ref{G2}), $p(x(\lambda,\alpha))$,  $h(x(\lambda,\alpha))$ and $g(x(\lambda,\alpha))$.
It leads to the following underestimation of $\mathcal{G}(\alpha)$ over $\alpha\in[\alpha_i,\alpha_{i+1}]$:
\begin{eqnarray}
\underline{\mathcal{G}}(\alpha):=
\max_{\lambda\in\Bbb R} &\min\limits_{y_p,y_h,y_g}&   \frac{y_p}{\alpha}+y_h-\lambda y_g+\alpha\lambda\label{lp:1}\\
&{\rm s.t.}&\frac{y_p}{\alpha_i}+y_h - \lambda(\alpha_i) y_g+\alpha_i \lambda(\alpha_i) \geq \mathcal{G}(\alpha_i),  \label{lp:2}\\
&&   \frac{y_p}{\alpha_{i+1}}+y_h- \lambda(\alpha_{i+1}) y_g+\alpha_{i+1} \lambda(\alpha_{i+1}) \geq \mathcal{G}(\alpha_{i+1}).\label{lp:3}
\end{eqnarray}
The inner optimization problem of (\ref{lp:1})-(\ref{lp:3})  in terms of $(y_p,y_h,y_g)$ is a linear program and hence it is equivalent to its dual maximization problem. Thus, we can rewrite the underestimation (\ref{lp:1})-(\ref{lp:3}) as a double maximization problem:
\begin{eqnarray}
\max_{\lambda\in\Bbb R}&\max\limits_{\mu_1,\mu_2}&  \mu_1(\mathcal{G}(\alpha_i)-\alpha_i \lambda(\alpha_i))+\mu_2(\mathcal{G}(\alpha_{i+1})-\alpha_{i+1} \lambda(\alpha_{i+1}))+\alpha\lambda \label{u0}\\
&{\rm s.t.}& \frac{1}{\alpha_i}\mu_1+\frac{1}{\alpha_{i+1}}\mu_2=\frac{1}{\alpha},  \label{u1}\\
& & \mu_1+\mu_2=1, \label{u2}\\
& &  \lambda(\alpha_i)\mu_1+ \lambda(\alpha_{i+1})\mu_2 = \lambda, \label{u3}\\
& &  \mu_1,\mu_2 \geq 0, \label{u4}
\end{eqnarray}
which can be recast as a standard optimization problem of maximizing (\ref{u0}) subject to (\ref{u1})-(\ref{u4}) with respect to $\lambda$, $\mu_1$ and $\mu_2$. It follows from $\alpha_i<\alpha_{i+1}$ that $\mu_1$ and $\mu_2$ can be uniquely solved by
the equalities (\ref{u1})-(\ref{u2}), that is,
\begin{equation}
\mu_1= \frac{\alpha_i(\alpha_{i+1}-\alpha)}{\alpha(\alpha_{i+1}-\alpha_i)},~  \mu_2= \frac{\alpha_{i+1}(\alpha-\alpha_i)}{\alpha(\alpha_{i+1}-\alpha_i)}. \label{u1u2}
\end{equation}
Moreover, for the solutions $\mu_1$ and $\mu_2$ of (\ref{u1u2}), the constraint (\ref{u4}) holds  as
$1\le\alpha_i<\alpha_{i+1}$ and $\alpha\in[\alpha_i,\alpha_{i+1}]$. Substituting the solutions $\mu_1$ and $\mu_2$ (\ref{u1u2}) into the equality constraint (\ref{u3}), we obtain
\begin{equation}
\lambda= \frac{\alpha_i \alpha_{i+1}}{\alpha_{i+1}-\alpha_i}(\lambda(\alpha_i) - \lambda(\alpha_{i+1})) \frac{1}{\alpha}+ \frac{1}{\alpha_{i+1}-\alpha_i} (\lambda(\alpha_{i+1}) \alpha_{i+1} - \lambda(\alpha_i) \alpha_i). \label{lambd}
\end{equation}
Therefore, the optimization problem (\ref{u0})-(\ref{u4}) has been explicitly solved.
Plugging (\ref{u1u2}) and (\ref{lambd}) in (\ref{u0}) yields a closed-form expression of the underestimation function $\underline{\mathcal{G}}(\alpha)$:
\begin{equation}
 \underline{\mathcal{G}}(\alpha)=
c_1 \alpha+\frac{c_2}{\alpha}+c_3,     \label{newub}
\end{equation}
where the constant coefficients are defined as
\begin{eqnarray}
c_1&=&\frac{\alpha_{i+1}\lambda(\alpha_{i+1})  - \alpha_i\lambda(\alpha_i)}{\alpha_{i+1}-\alpha_i},\label{c1}\\
c_2&=&\alpha_i \alpha_{i+1}\left(c_1-
\frac{\mathcal{G}(\alpha_{i+1})-\mathcal{G}(\alpha_i)}{\alpha_{i+1}-\alpha_i}
\right),\label{c2}\\
c_3&=&\frac{ \alpha_{i+1}\mathcal{G}(\alpha_{i+1})-  \alpha_i\mathcal{G}(\alpha_{i})}{\alpha_{i+1}-\alpha_i}-
c_1(\alpha_{i+1}+\alpha_{i})
.\label{c3}
\end{eqnarray}
The underestimation $\underline{\mathcal{G}}(\alpha)$ is tight at the two endpoints as we can verify that
\begin{equation}
\underline{\mathcal{G}}(\alpha_{i+1})= \mathcal{G}(\alpha_{i+1}),~
\underline{\mathcal{G}}(\alpha_{i})=\mathcal{G}(\alpha_{i}). \label{alphai}
\end{equation}
Then, the minimum of $\underline{\mathcal{G}}(\alpha)$ over $[\alpha_{i},\alpha_{i+1}]$ provides
a lower bound of (\ref{newb0}). By simple computation,  we have
\[
\min_{\alpha\in[\alpha_i,\alpha_{i+1}]}\underline{\mathcal{G}}(\alpha)=\left\{\begin{array}{ll}
2\sqrt{ c_1c_2}+c_3,&{\rm if}~c_1>0, c_2>0, \alpha_i<\frac{\sqrt{c_2}}{\sqrt{c_1}}<\alpha_{i+1},\\
\underline{\mathcal{G}}(\alpha_{i+1}),& {\rm if}~c_1>0, c_2>0,
\alpha_{i+1}\le \frac{\sqrt{c_2}}{\sqrt{c_1}},\\
\underline{\mathcal{G}}(\alpha_{i}),& {\rm if}~c_1>0, c_2>0,
\alpha_{i}\ge \frac{\sqrt{c_2}}{\sqrt{c_1}},\\
\underline{\mathcal{G}}(\alpha_i),& {\rm if}~c_1>0,c_2\le0, \\
\underline{\mathcal{G}}(\alpha_{i+1}),& {\rm if}~c_1\le0,c_2>0, \\
\min\left\{\underline{\mathcal{G}}(\alpha_{i+1}),
\underline{\mathcal{G}}(\alpha_{i})\right\},& {\rm if}~ c_1\le0,c_2\le0.\\
\end{array}\right.
\]
As a summary, we have the following result.
\begin{thm}\label{main}
Let $c_1,c_2$ and $c_3$ be defined in (\ref{c1})-(\ref{c3}), respectively. If
\begin{equation}
c_1>0, ~c_2>0, ~\widetilde{\alpha}:=\sqrt{\frac{c_2}{c_1}}\in(\alpha_i,\alpha_{i+1}), \label{cod}
\end{equation}
then  we have
\begin{equation}
\min_{\alpha\in[\alpha_i,\alpha_{i+1}]}\mathcal{G}(\alpha)
\ge \min_{\alpha\in[\alpha_i,\alpha_{i+1}]}\underline{\mathcal{G}}(\alpha)=
2\sqrt{ c_1c_2}+c_3, \nonumber
\end{equation}
where $\widetilde{\alpha}$ is the unique minimizer of $\underline{\mathcal{G}}(\alpha)$  over $[\alpha_i,\alpha_{i+1}]$.
If (\ref{cod}) does not hold, we have
\begin{equation}
\min_{\alpha\in[\alpha_i,\alpha_{i+1}]}\mathcal{G}(\alpha)
=\min\left\{\mathcal{G}(\alpha_{i}),\mathcal{G}(\alpha_{i+1})\right\}.\nonumber
\end{equation}
\end{thm}

\subsection{A new branch-and-bound algorithm(BTD Algorithm)}
In this subsection, we employ a branch-and-bound algorithm to solve $({\rm PM})$ (\ref{1dim}) based on the above novel underestimation. For the rule of branching, we adopt the $\omega$-subdivision approach, i.e., we select $\widetilde{\alpha}$ defined in (\ref{cod}), the minimizer of the underestimating function $\underline{\mathcal{G}}(\alpha)$, to subdivide the current interval $[\alpha_{i},\alpha_{i+1}]$.  The whole algorithm is listed as follows.

\begin{center}
\fbox{\shortstack[l]{
{\bf BTD Algorithm}\\
1.~Input: $A\in \Bbb R^{m\times n}$, $b\in \Bbb R^m$, $L\in \Bbb R^{k\times n}$, $\rho> 0$,\\ ~~~~~~~~~~~~and $\epsilon>0$: the tolerance parameter.\\
2.~If $b=0$, let $\alpha^*=1$ and go to Step 8. Otherwise, goto Step 3. \\
3.~If $A^Tb\neq 0$, set $\alpha_{1}$ as the lower bound given in Theorem \ref{thm:lowb},\\
~~~else set $\alpha_{1}=\frac{\|b\|^2}{\|b\|^2-\epsilon}$ according to Proposition \ref{prop1}. \\
~~~Let $\alpha_{2}$ be the upper bound given in Theorem \ref{thm:newub}.\\
~~~For $i=1,2$ compute $\mathcal{G}(\alpha_{i})$ and Lagrange multiplier
$\lambda_{i}$ for $\|x\|^2=\alpha-1$.\\
~~~Set $k=2$ (the number of functional evaluations (i.e., iterations)).\\
~~~Set $UB=\mathcal{G}(\alpha_{1})$,  $\alpha^*=\alpha_1$ and $T=\emptyset$. \\
~~~If $\mathcal{G}(\alpha_{2})<UB$, update $UB=\mathcal{G}(\alpha_{2})$ and $\alpha^*=\alpha_2$.\\
4.~If (\ref{cod}) does not hold for $[\alpha_i,\alpha_{i+1}]=:[\alpha_1,\alpha_{2}]$, go to Step 8. \\
~~~Otherwise, set $\tilde \alpha$ as in (\ref{cod}) where $[\alpha_i,\alpha_{i+1}]=:[\alpha_1,\alpha_{2}]$.\\
~~~Evaluate $\mathcal{G}(\tilde\alpha)$ and Lagrange multiplier $\tilde\lambda$ for $\|x\|^2=\alpha-1$.\\
~~~If $\mathcal{G}(\tilde\alpha) < UB$, update $UB= \mathcal{G}(\tilde\alpha)$ and $\alpha^*=\tilde \alpha$.\\
5.~Use Theorem \ref{main} to compute the lower bounds over $[\alpha_i,\alpha_{i+1}]=:[\alpha_{1},\tilde\alpha]$ \\
~~~and $[\alpha_i,\alpha_{i+1}]=:[\tilde \alpha,\alpha_{2}]$, denoted by $LB_1$ and $LB_2$, respectively.\\
~~~If $LB_1<UB-\epsilon$, update $T:=T\cup\{(LB_1,\alpha_{1},\tilde \alpha)\}$.\\
~~~If $LB_2<UB-\epsilon$, update $T:=T\cup\{(LB_2,\tilde\alpha,\alpha_{2})\}$.\\
~~~Update $k:=k+1$.\\
6.~If $T=\emptyset$, go to Step 7. Otherwise, find $(LB^*,\alpha_1,\alpha_2):= \arg \min\limits_{(t,*,*)\in T} t$.\\
~~~If $LB^*\geq UB-\epsilon$, go to Step 7, otherwise, update   \\
~~~$T:=T\setminus\{(LB^*,\alpha_1,\alpha_2)\}$ and go to Step 4.\\
7.~Output $\alpha^*$: an approximately optimal solution of $({\rm PM})$ (\ref{1dim}).\\
8.~Output $\alpha^*$: an exact global minimizer of $({\rm PM})$ (\ref{1dim}).
}}
\end{center}

Since there is no detailed data of the counterexample in \cite{B06}, in the following we give a new exceptional example where $G{\alpha}$ is not unimodal.

\begin{exam}\label{exm}
Let $m = n =2$, $k = 1$ and
\[
A=\left(\begin{matrix} 0.4 & 0.8 \\ 0.2 & 1 \end{matrix}\right),~b=\left(\begin{matrix} 0.1 \\  0.5 \end{matrix}\right),~ L=\left(\begin{matrix}  0.1&  0.8\end{matrix}\right),~\rho=0.5.
\]
With the same setting $\epsilon_1=10^{-1},~\epsilon_2=10^{-6}$ as given in \cite{B06},
after $35$ iterations, Algorithm TRTLSG  finds a local non-global minimizer $\widetilde{x}=(3.2209,-0.4897)^T$ with the objective function value $0.0673$ and $\widetilde{\alpha}=\|\widetilde{x}\|^2+1\approx 11.6140$. Actually, the global minimizer of $(\rm PM)$ (\ref{1dim}) is  $\alpha^*\approx 1.6300$ and the corresponding objective value is $v(\rm PM)\approx 0.0634$. The function $\mathcal{G}(\alpha)$ for this example is plotted in Figure \ref{fig1}, see Section 5.
\end{exam}

We show the details of applying our new algorithm to solve Example \ref{exm} with the setting $\epsilon=10^{-6}$. It follows from Theorems \ref{thm:lowb} and \ref{thm:newub} that
\[
\alpha_1=\alpha_{\min}=1.0266,~  \alpha_2=\alpha_{\max}=3355.5794.
\]
As a contrast, the upper bound  (\ref{ubb2}) given in \cite{B06} is
$17551.0566$.
Following the $\omega$-subdivision approach, the first subdividing point (\ref{cod}) is given by
\[
\alpha_3=\widetilde{\alpha}={\rm arg}\min_{\alpha\in[\alpha_1,\alpha_{2}]}\underline{\mathcal{G}}(\alpha)=59.1724.
\]
The next $12$ iterations are plotted in Figure \ref{fig1} and then the stopping criterion is reached. It returns a global approximation solution $x^*=(-0.6541,0.4496)^T$ with $\alpha^*=\|x^*\|^2+1\approx1.6300$.  It is observed that the algorithm based on the $\omega$-subdivision is much more efficient than that based on bisection. Moreover, in each iteration, our new lower bound is tight for one of the two subintervals divided from
 the current interval. Consequently, there is
 no  need to  subdivide this subinterval  in the following iterations.

\begin{figure}[h]
\centering
  \includegraphics[width=9cm]{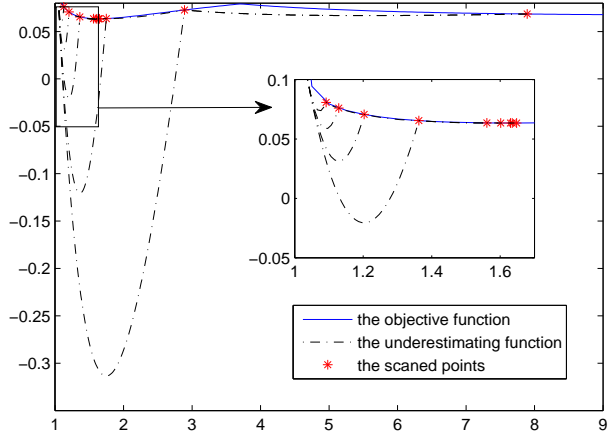}
  \caption{The last $12$ iterations of our new algorithm for solving Example \ref{exm}.
 }
  \label{fig1}
\end{figure}

Let $\alpha^*$ be the solution obtained by our new algorithm. It holds that
\begin{equation}
v({\rm P_G})\le \mathcal{G}(\alpha^*) \le v({\rm P_G})+\epsilon.\label{apps}
\end{equation}
Throughout this paper, any $\alpha^*\ge1$  satisfying (\ref{apps}) is called a global $\epsilon$-approximation solution of $({\rm P_G})$.

In order to study the worst-case computational complexity of our new algorithm, we need the following lemma.
\begin{lem}\label{lem2}
Let $\lambda(\alpha)$ be the Lagrangian multiplier of (TRS) (\ref{1dim2}) (i.e., the $\lambda$-solution of the KKT system (\ref{KKT1})-(\ref{KKT3})). Then, if $\alpha_{\min}>1$,  $\lambda(\alpha)$ is bounded over $[\alpha_{\min},\alpha_{\max}]$:
\begin{equation}
|\lambda(\alpha)| \le U:=\frac{\|A^Tb\|}{\alpha_{\min}\sqrt{\alpha_{\min}-1}}+ \lambda_{\min}\left(\frac{1}{\alpha_{\min}}A^TA+ \rho L^TL\right).\label{lam:ud}
\end{equation}
\end{lem}
\begin{proof}
It follows from the KKT system (\ref{KKT1})-(\ref{KKT3}) that if
$\lambda(\alpha)\neq \lambda_{\min}(Q_{\alpha})$ then
$\lambda(\alpha)< \lambda_{\min}(Q_{\alpha})$ and
\begin{equation}
\|(Q_{\alpha}-\lambda(\alpha)I)^{-1}f_{\alpha}\|^2= \alpha-1. \label{seq:1}
\end{equation}
Notice that
\begin{eqnarray}
\|(Q_{\alpha}-\lambda(\alpha)I)^{-1}f_{\alpha}\|^2
&\le& \lambda_{\max}^2\left(\left(Q_{\alpha}-\lambda(\alpha)I\right)^{-1}\right)\|f_{\alpha}\|^2 \nonumber\\
&=& 
\frac{\|f_{\alpha}\|^2}{\left(\lambda_{\min}\left(Q_{\alpha}\right)-\lambda(\alpha)\right)^2}.\label{seq:2}
\end{eqnarray}
Plugging (\ref{seq:2}) in (\ref{seq:1}) yields
\[
\frac{\|f_{\alpha}\|^2}{\left(\lambda_{\min}\left(Q_{\alpha}\right)-\lambda(\alpha)\right)^2} \ge \alpha-1,
\]
which further implies that
\begin{equation}
|\lambda(\alpha)|\le \frac{\|f_{\alpha}\|}{\sqrt{\alpha-1}}+ \lambda_{\min}\left(Q_{\alpha}\right).\label{lbd2}
\end{equation}
Notice that the inequality (\ref{lbd2}) trivially holds true for the other case $\lambda(\alpha)= \lambda_{\min}(Q_{\alpha})$.
Then, according to (\ref{lbd2}) and the definitions of $Q_{\alpha}$  and $f_{\alpha}$ (\ref{Qf}), we obtain the upper bound (\ref{lam:ud}) over the interval $[\alpha_{\min},\alpha_{\max}]$.
\end{proof}

\begin{thm}\label{complex}
If $A^Tb\neq 0$, our new algorithm finds a global $\epsilon$-approximation solution of $({\rm P_G})$ (\ref{1dim}) in  at most
\begin{equation}
\left\lceil\frac{4U\alpha_{\max}^2(\alpha_{\max}-\alpha_{\min})}{ \alpha_{\min}^2~\epsilon}\right\rceil\label{comp:1}
\end{equation}
iterations,
where  $U$ is defined in (\ref{lam:ud}), $\alpha_{\min}>1$ and $\alpha_{\max}$ are constant real numbers defined in
Theorems \ref{thm:lowb} and \ref{thm:newub}, respectively. Moreover, suppose the assumption (\ref{uniq}) holds for all $\alpha>1$,
in order to find a global $\epsilon$-approximation solution of $({\rm P_G})$, our new algorithm requires  at most
\begin{equation}
\left\lceil\frac{2\widetilde{U}\sqrt{\alpha_{\max}}
(\alpha_{\max}-\alpha_{\min})}{ \alpha_{\min}~\cdot \sqrt{\epsilon}}\right\rceil\label{comp:2}
\end{equation}
iterations, where
\begin{equation}
\widetilde{U}=\max_{\alpha\in[\alpha_{\min},
\alpha_{\max}]}\lambda(\alpha)+\alpha\lambda'(\alpha), \label{Util}
\end{equation}
is a well-defined finite number and $\lambda(\alpha)$ is the $\lambda$-solution of (\ref{KKT1})-(\ref{KKT3}).
\end{thm}
\begin{proof}
Suppose $(LB,\alpha_i,\alpha_{i+1})\in T$ is selected to subdivide in the current iteration of our new algorithm. Then, we have $LB=LB^*$. Without loss of generality, we assume that
(\ref{cod}) holds in the interval $[\alpha_i,\alpha_{i+1}]$, since otherwise, it follows from  Theorem \ref{main} that $LB=UB$ and hence the algorithm has to stop.

The condition (\ref{cod}) implies that the underestimating function $\underline{\mathcal{G}}(\alpha)$ (\ref{newub}) is convex. Therefore, for any $\alpha\in[\alpha_{\min},\alpha_{\max}]$, we have
\begin{eqnarray}
\underline{\mathcal{G}}(\alpha)&\ge& \underline{\mathcal{G}}(\alpha_i)+\underline{\mathcal{G}}'(\alpha_i)(\alpha-\alpha_i)\nonumber\\
&
=&\mathcal{G}(\alpha_i)+\left(c_1-\frac{c_2}{\alpha_i^2}\right)(\alpha-\alpha_i)\nonumber\\
&
\ge&\mathcal{G}(\alpha_i)+\left(c_1-\frac{c_1\alpha^2_{i+1}}{\alpha_i^2}\right)
(\alpha-\alpha_i)\nonumber\\
&\ge&\mathcal{G}(\alpha_i)- \frac{\alpha^2_{i+1}-\alpha_i^2}{\alpha_i^2}
c_1(\alpha_{i+1}-\alpha_i),\label{G4}
\end{eqnarray}
where the first equality follows from (\ref{alphai}) and the second inequality holds due to the third inequality of (\ref{cod}).

According to the definition (\ref{c1}) and Lemma \ref{lem2}, we have
\begin{equation}
c_1(\alpha_{i+1}-\alpha_i)=\alpha_{i+1}\lambda(\alpha_{i+1})  - \alpha_i\lambda(\alpha_i)
\le (\alpha_{i+1}+\alpha_{i})U. \label{c1b}
\end{equation}
By substituting (\ref{c1b}) into (\ref{G4}), we obtain
\begin{eqnarray}
LB^*=\min_{\alpha\in[\alpha_i,\alpha_{i+1}]}\underline{\mathcal{G}}(\alpha)
&\ge& \mathcal{G}(\alpha_i)- \frac{(\alpha_{i+1}+\alpha_i)^2(\alpha_{i+1}-\alpha_i)}{\alpha_i^2}
U\nonumber
\\&\ge&
UB- \frac{4\alpha_{\max}^2(\alpha_{i+1}-\alpha_i)}{\alpha_{\min}^2}
U.\nonumber
\end{eqnarray}
Consequently, the stopping criterion $LB^*>UB- \epsilon$ is reached if
\[
\alpha_{i+1}-\alpha_i
< \frac{\alpha_{\min}^2}{4U\alpha_{\max}^2}\cdot \epsilon.
\]
Therefore, the number of the iterations of our new algorithm can not exceed the upper bound (\ref{comp:1}).

It has been shown in the first part of the proof of Theorem \ref{grad} (see \cite{B06}) that, under the assumption that (\ref{uniq}) holds for all $\alpha>1$, $\lambda(\alpha)$ is differentiable of any order. Therefore, $\widetilde{U}$ (\ref{Util}) is well defined and $\widetilde{U}<+\infty$.  Applying the mean-value theorem to the definition of $c_1$ (\ref{c1}), we have
\[
c_1=\frac{\alpha_{i+1}\lambda(\alpha_{i+1})-\alpha_i\lambda(\alpha_i)}{\alpha_{i+1}-\alpha_i}
= \left(\alpha\lambda(\alpha)\right)'|_{\alpha=\xi} =\lambda(\xi)+\xi\lambda'(\xi)\le \widetilde{U},
\]
where $\xi\in(\alpha_i,\alpha_{i+1})$.

Then, it follows from (\ref{G4}) that
\begin{eqnarray}
LB^*=\min_{\alpha\in[\alpha_i,\alpha_{i+1}]}\underline{\mathcal{G}}(\alpha)
&\ge&\mathcal{G}(\alpha_i)- \frac{(\alpha_{i+1}+\alpha_i)
(\alpha_{i+1}-\alpha_i)^2}{\alpha_i^2}
\widetilde{U} \nonumber\\
&\ge&UB- \frac{2\alpha_{\max}\widetilde{U}
}{\alpha_{\min}^2}(\alpha_{i+1}-\alpha_i)^2.\nonumber
\end{eqnarray}
Then, if
\[
\alpha_{i+1}-\alpha_i
< \frac{\alpha_{\min}}{\sqrt{2\widetilde{U}\alpha_{\max}}}
\cdot \sqrt{\epsilon},
\]
the stopping criterion $LB^*>UB- \epsilon$ is reached. Consequently, (\ref{comp:2}) gives the maximal number of the iterations of our new algorithm in the worst case.
\end{proof}

\begin{cor}
Suppose $A^Tb= 0$ and $b\neq 0$. For any $\epsilon\in(0,\|b\|^2)$,
our new algorithm finds a global $\epsilon$-approximation  solution of $({\rm P_G})$ (\ref{1dim}) in  at most
\begin{equation}
\left\lceil
\frac{4\alpha_{\max}^2(\alpha_{\max}-1)\lambda_{\min}\left( A^TA+ \rho L^TL\right)}{\epsilon}
\right\rceil \label{comp:3}
\end{equation}
iterations, where
$\alpha_{\max}$ is defined in Theorem \ref{thm:newub}.
\end{cor}
\begin{proof}
Under the assumption $A^Tb= 0$ and $b\neq 0$, according to Proposition \ref{prop1},  we have $\alpha_{\min}=\frac{\|b\|^2}{\|b\|^2-\epsilon}>1$.
Then, Lemma \ref{lem2} and Theorem \ref{complex} hold true. It follows from (\ref{lam:ud}) that
\begin{eqnarray}
 U&=&\frac{\|A^Tb\|}{\alpha_{\min}\sqrt{\alpha_{\min}-1}}+ \lambda_{\min}\left(\frac{1}{\alpha_{\min}}A^TA+ \rho L^TL\right)\nonumber\\
 &=&\lambda_{\min}\left(\frac{1}{\alpha_{\min}}A^TA+ \rho L^TL\right)\nonumber\\
  &\le& \lambda_{\min}\left( A^TA+ \rho L^TL\right).\nonumber
\end{eqnarray}
According to Theorem \ref{complex} and the following inequality
\[
\frac{4U\alpha_{\max}^2(\alpha_{\max}-\alpha_{\min})}{ \alpha_{\min}^2~\epsilon}
\le  \frac{4\alpha_{\max}^2(\alpha_{\max}-1)\lambda_{\min}\left( A^TA+ \rho L^TL\right)}{\epsilon},
\]
the proof is complete.
\end{proof}
\begin{rem}
The worst-case computational complexity (\ref{comp:3}) can not be similarly reduced to  $O(1/\sqrt{\epsilon})$ as in (\ref{comp:2}), since for any $\alpha>1$, the
assumption (\ref{uniq}) can not hold true for the case $A^Tb= 0$.
\end{rem}

\section{Numerical experiments}
In this section, we numerically compare the computational efficiency of the improved version of the bisection-based Algorithm TRTLSG \cite{B06} (which is improved by strengthening the lower and upper bounds on the norm of the optimal solution, see Section 3) and our new branch-and-bound algorithm(denoted by BTD Algorithm). Since the stopping criterion in Step $3$ of Algorithm TRTLSG \cite{B06} is different from that of our global optimization algorithm,
for the sake of fairness, we replace the original simple stopping criterion $|\alpha_{\max}- \alpha_{\min}| > \epsilon_2$  with
\begin{equation}
\mathcal{G}(\alpha_{\max})\leq LB^*+\epsilon,\label{stop}
\end{equation}
where $LB^*\in [\mathcal{G}(\alpha^*)-\epsilon,\mathcal{G}(\alpha^*)]$ is a  lower approximation of the optimal value $\mathcal{G}(\alpha^*)$ obtained by calling our new global optimization algorithm in advance.

We numerically test two examples. The first one
is taken from Hansen's Regularization Tools \cite{H1994},
where the function $shaw$ is used to generate the matrix
$A_{\rm true}\in\Bbb R^{n\times n}$, the vector $b_{\rm true}\in\Bbb R^n$
and the true solution $x_{\rm true}\in\Bbb R^n$, i.e.,
we have $A_{\rm true}x_{\rm true}=b_{\rm ture}$.
Then, we add the white noise of level $\sigma=0.05$, i.e.,
$A=A_{\rm true}+\sigma E$, $b=b_{\rm true}+\sigma e$,
where $E$ and $e$ are generated from a standard normal
distribution. In our experiments,  the dimension  $n$ varies
from $20$ to $5000$.

The second one is an image deblurring example of a fixed dimension $n=1024$, see \cite{B06,BBT2006}. We generate the atmospheric turbulence blur matrix $A_{\rm true}\in\Bbb R^{n\times n}$ by implementing $blur(n,3)$, which is taken from \cite{H1994}. The true solution $x_{\rm true}\in \Bbb R^{n}$ is obtained by stacking the columns of $X\in\Bbb R^{32\times 32}$ one underneath the other and then normalizing it so that $\|x_{\rm true}\|=1$, where  $X\in\Bbb R^{32\times 32}$ is the following two dimensional image:
\[
X(z_1,z_2)=\sum_{l=1}^3 a_l\cos(w_{l,1}z_1+w_{l,2}z_2+\phi_l),1\leq z_1,z_2\leq 32,
\]
with the coefficients being given in Table 1 of \cite{B06}.  Let $b_{\rm true}=A_{\rm true}x_{\rm true}$. Then, the white noise is added, i.e.,  $A=A_{\rm true}+\sigma E$, $b=b_{\rm true}+\sigma e$, where $E$ and $e$ are generated from a standard normal distribution.  In our experiments, we let the level of the noise $\sigma$ vary in $\{0.01,0.03,0.05, 0.08,0.1,0.3,0.5,0.8,1.3,1.5,1.8,2.0\}$.

For the regularization matrix of the first example, we take $L=get\_l(n,1)$, which is given in  \cite{H1994}. For the second example, as in \cite{BBT2006}, we set the regularization matrix $L$ as the discrete approximation of the Laplace operator, which is standard in image processing \cite{J89}.
The regularization parameter $\rho$ is selected by using the L-curve method \cite{H1993}. It corresponds to the  L-shaped corner of  the norm $\|Lx\|^2$ versus the fractional residual $\|Ax-b\|^2/(\|x\|^2+1)$ for a various number of regularization parameters.

All the experiments are carried out in MATLAB R2014a and run
on a server with 2.6 GHz dual-core processor and 32 GB RAM. We set the tolerance parameter $\epsilon=10^{-6}$ for all the three algorithms.
For each setting of the dimension or the level of noise in the above two examples, we independently and randomly generate $10$ instances and then run the three algorithms. We report in Tables \ref{tab:2} and \ref{tab:3} the average of the numerical results for the $10$ times running, where the average computational time is recorded in seconds and  the symbol `\#iter' denotes the average of the number of iterations, i.e., the number of evaluating (TRS).

The numerical results demonstrate that, in most cases, our global optimization algorithm outperforms the improved version of the heuristic Algorithm TRTLSG \cite{B06}. Moreover, the larger the dimension or the level of noise is, the faster our global algorithms performs.  It is worth noting that with the modified stopping criterion (\ref{stop}), the improved Algorithm TRTLSG \cite{B06} requires much fewer iterations as the objective $\mathcal{G}(\alpha)$ is quite flat around the optimal solution $\alpha^*$. So, it is more time-consuming if the original simple stopping criterion $|\alpha_{\max}- \alpha_{\min}| > \epsilon_2$ is used.
It is observed that the number of the iterations of the improved Algorithm TRTLSG (though slightly) increases with the increase of either the dimension or the level of noise. However, for all instances we have tested, the number of the iterations of our new global optimization algorithm is never larger than twenty and seems to be independent of the dimension and the level of noise.

\begin{table}[h]\centering
\caption{The average of the numerical results for ten times solving
 the first example with different dimension $n$. } \label{tab:2}
\begin{center}
\begin{tabular}{rrrrr} \hline
  &   \multicolumn{2}{c}{Algorithm TRTLSG} & \multicolumn{2}{c}{Algorithm BTD} \\
  \cline{2-3}\cline{4-5}
n &\#~iter & time (s)&\#~iter & time (s) \\
\hline
20 & 16.0 & 0.02 & 17.0 & 0.02   \\
50   & 18.3 & 0.03 & 15.5& 0.03     \\
100   & 18.7 & 0.09 & 15.5 & 0.08   \\
200   & 18.9& 0.26  & 16.5 & 0.25   \\
500   & 20.0 & 2.81 & 16.8 & 2.64    \\
1000  & 20.5& 10.49  & 16.1 & 9.10   \\
1200  & 20.4 & 15.19 & 15.6  & 12.69  \\
1500  & 21.1 & 24.07 & 18.0 & 22.40   \\
1800  & 21.2& 35.97  & 17.8  & 33.67  \\
2000  & 20.8 & 43.88 & 17.8 & 43.10   \\
2500  & 20.7& 72.51  & 17.5 & 68.84    \\
3000  & 21.8 & 125.16 & 16.2 & 102.76   \\
4000  & 20.2& 255.86  & 14.0 & 202.95    \\
5000  & 20.0 & 448.39 & 14.5 & 366.50   \\
 \hline
 \hline
 \end{tabular}
 \end{center}
\end{table}

\begin{table}[h]\centering
\caption{The average of the numerical results for ten times solving
 the second example with a fixed dimension $n=1024$ and different level of noise $\sigma$.} \label{tab:3}
\begin{center}
\begin{tabular}{rrrrr} \hline
  &   \multicolumn{2}{c}{Algorithm TRTLSG} & \multicolumn{2}{c}{Algorithm BTD} \\
  \cline{2-3}\cline{4-5}
 $\sigma$  &\#~iter & time (s)&\#~iter & time (s) \\
\hline
0.01   & 17.2& 9.46  & 14.4& 8.60    \\
0.03   & 21.9& 12.33  & 16.6& 10.11    \\
0.05   & 16.2 & 8.91 & 17.0  & 10.52  \\
0.08  & 18.0& 10.04  & 17.0 & 10.68   \\
0.1   & 19.8 & 11.25 & 18.4 & 11.56   \\
0.3    & 29.4& 17.65   & 17.0& 10.94    \\
0.5    & 30.8 & 18.59 & 17.4& 11.19    \\
0.8    & 30.7& 20.52  & 15.9& 12.06    \\
1.0   & 31.4& 21.08  & 15.4 & 11.71   \\
1.3    & 32.2& 21.56  & 15.6 & 12.04   \\
1.5    & 32.0& 21.96   & 15.6 & 12.33   \\
1.8   & 33.6 & 23.11 & 16.1 & 13.04   \\
2.0   & 33.7 & 23.52  & 16.0& 13.03    \\
 \hline
 \hline
 \end{tabular}
 \end{center}
\end{table}

\section{Conclusions}

The total least squares problem with the general Tikhonov regularization (TRTLS) is a non-convex optimization problem with local non-global minimizers. It can be reformulated as a problem of minimizing the one-dimensional function $\mathcal{G}(\alpha)$ over an interval, where $\mathcal{G}(\alpha)$ is evaluated by solving an $n$-dimensional trust region subproblem. In literature, there is an efficient bisection-based heuristic algorithm for solving (TRTLS), denoted by Algorithm TRTLSG. It converges to the global optimal solution except for some exceptional examples with non-unimodal $\mathcal{G}(\alpha)$. In this paper, we firstly improve the lower and upper bounds on the norm of the globally optimal solution. It helps to greatly improve  the efficiency of Algorithm TRTLSG. For the global optimization of (TRTLS), we employ the adaptive branch-and-bound algorithm, based on a novel two-layer dual approach for underestimating $\mathcal{G}(\alpha)$ over any given interval. Our new algorithm(Algorithm BTD) guarantees to find a global $\epsilon$-approximation  solution in at most $O(1/\epsilon)$ iterations and the computational effort in each iteration is $O(n^3\log(1/\epsilon))$.
Under the same assumptions as in Algorithm TRTLSG, the number of iterations of our new algorithm can be further reduced to $O(1/\sqrt{\epsilon})$.
In our experiments,  the practical iteration numbers are always less than twenty and seem to be independent of the dimension and the level of noise.
Numerical results demonstrate that our global optimization algorithm is even faster than the improved version of Algorithm TRTLSG, which is a bisection-based heuristic algorithm.
It is the future work to extend our novel underestimation approach to globally solve more  structured non-convex optimization problems.

\end{document}